\documentclass[12pt]{amsart}
\usepackage{amsmath,amssymb,amsbsy,amsfonts,amsthm,latexsym,
                        amsopn,amstext,amsxtra,euscript,amscd,mathrsfs,color,bm,mathtools}
                        
\usepackage{float} 
\usepackage[english]{babel}
\usepackage{url}
\usepackage[colorlinks,linkcolor=blue,anchorcolor=blue,citecolor=blue,backref=page]{hyperref}

\usepackage[margin=2.9cm]{geometry}

\newcounter{constant}
\newcommand{\newconstant}[1]{\refstepcounter{constant}\label{#1}}
\newcommand{\useconstant}[1]{C_{\ref{#1}}}
\newcommand{\defconstant}[1]{ \newconstant{c_#1}\expandafter\newcommand\csname co#1\endcsname{\useconstant{c_#1}} }

\usepackage{color}
\usepackage[norefs,nocites]{refcheck}

\renewcommand*{\backref}[1]{}
\renewcommand*{\backrefalt}[4]{%
    \ifcase #1 (Not cited.)%
    \or        (p.\,#2)%
    \else      (pp.\,#2)%
    \fi}

\begin{document}

\newtheorem{theorem}{Theorem}
\newtheorem{lemma}[theorem]{Lemma}
\newtheorem{example}[theorem]{Example}
\newtheorem{algol}{Algorithm}
\newtheorem{corollary}[theorem]{Corollary}
\newtheorem{prop}[theorem]{Proposition}
\newtheorem{definition}[theorem]{Definition}
\newtheorem{question}[theorem]{Question}
\newtheorem{problem}[theorem]{Problem}
\newtheorem{remark}[theorem]{Remark}
\newtheorem{conjecture}[theorem]{Conjecture}

\newcommand{\llrrparen}[1]{
  \left(\mkern-3mu\left(#1\right)\mkern-3mu\right)}

\newcommand{\commM}[1]{\marginpar{%
\begin{color}{red}
\vskip-\baselineskip 
\raggedright\footnotesize
\itshape\hrule \smallskip M: #1\par\smallskip\hrule\end{color}}}

\newcommand{\commA}[1]{\marginpar{%
\begin{color}{blue}
\vskip-\baselineskip 
\raggedright\footnotesize
\itshape\hrule \smallskip A: #1\par\smallskip\hrule\end{color}}}
\def\xxx{\vskip5pt\hrule\vskip5pt}


\def\cA{{\mathcal A}}
\def\cB{{\mathcal B}}
\def\cC{{\mathcal C}}
\def\cD{{\mathcal D}}
\def\cE{{\mathcal E}}
\def\cF{{\mathcal F}}
\def\cG{{\mathcal G}}
\def\cH{{\mathcal H}}
\def\cI{{\mathcal I}}
\def\cJ{{\mathcal J}}
\def\cK{{\mathcal K}}
\def\cL{{\mathcal L}}
\def\cM{{\mathcal M}}
\def\cN{{\mathcal N}}
\def\cO{{\mathcal O}}
\def\cP{{\mathcal P}}
\def\cQ{{\mathcal Q}}
\def\cR{{\mathcal R}}
\def\cS{{\mathcal S}}
\def\cT{{\mathcal T}}
\def\cU{{\mathcal U}}
\def\cV{{\mathcal V}}
\def\cW{{\mathcal W}}
\def\cX{{\mathcal X}}
\def\cY{{\mathcal Y}}
\def\cZ{{\mathcal Z}}

\def\C{\mathbb{C}}
\def\D{\mathbb{D}}
\def\I{\mathbb{I}}
\def\F{\mathbb{F}}
\def\H{\mathbb{H}}
\def\K{\mathbb{K}}
\def\G{\mathbb{G}}
\def\Z{\mathbb{Z}}
\def\R{\mathbb{R}}
\def\Q{\mathbb{Q}}
\def\N{\mathbb{N}}
\def\M{\textsf{M}}
\def\U{\mathbb{U}}
\def\P{\mathbb{P}}
\def\A{\mathbb{A}}
\def\p{\mathfrak{p}}
\def\n{\mathfrak{n}}
\def\X{\mathcal{X}}
\def\x{\textrm{\bf x}}
\def\w{\textrm{\bf w}}
\def\z{\mathrm{z}}
\def\ord{\mathrm{ord}}
\def\ovQ{\overline{\Q}}
\def\rank#1{\mathrm{rank}#1}
\def\wf{\widetilde{f}}
\def\wg{\widetilde{g}}
\def\comp{\hskip -2.5pt \circ  \hskip -2.5pt}
\def\({\left(}
\def\){\right)}
\def\[{\left[}
\def\]{\right]}
\def\<{\langle}
\def\>{\rangle}
\newcommand*\diff{\mathop{}\!\mathrm{d}}
\def\Lip{\mathrm{Lip}}

\def\BP{\mathbb{P}_{\mathrm{Berk}}^1}
\def\BA{\mathbb{A}_{\mathrm{Berk}}^1}
\def\cJBv{\mathcal{J}_{v,\mathrm{Berk}}}
\def\cFBv{\mathcal{F}_{v,\mathrm{Berk}}}

\def\gen#1{{\left\langle#1\right\rangle}}
\def\genp#1{{\left\langle#1\right\rangle}_p}
\def\genPs{{\left\langle P_1, \ldots, P_s\right\rangle}}
\def\genPsp{{\left\langle P_1, \ldots, P_s\right\rangle}_p}

\def\e{e}

\def\eq{\e_q}
\def\fh{{\mathfrak h}}

\def\lcm{{\mathrm{lcm}}\,}

\def\l({\left(}
\def\r){\right)}
\def\fl#1{\left\lfloor#1\right\rfloor}
\def\rf#1{\left\lceil#1\right\rceil}
\def\mand{\qquad\mbox{and}\qquad}

\def\jt{\tilde\jmath}
\def\ellmax{\ell_{\rm max}}
\def\llog{\log\log}

\def\diam{{\mathrm{diam}}}
\def\m{{\mathfrak{m}}}
\def\ud{{\mathrm{d}}}
\def\ch{\hat{h}}
\def\GL{{\rm GL}}
\def\Orb{\mathrm{Orb}}
\def\Per{\mathrm{Per}}
\def\larg{\mathrm{larg}}
\def\Preper{\mathrm{Preper}}
\def \S{\mathcal{S}}
\def\vec#1{\mathbf{#1}}
\def\ov#1{{\overline{#1}}}
\def\Gal{{\rm Gal}}

\newcommand{\bfalpha}{{\boldsymbol{\alpha}}}
\newcommand{\bfomega}{{\boldsymbol{\omega}}}

\newcommand{\canheight}[1]{{\hat h_{#1}}}
\newcommand{\loccanheight}[2]{{\hat h_{#2,#1}}}
\newcommand{\relsupreper}[3]{{\cP_{#1,#2}(#3)}}
\newcommand{\disunit}[3]{{\cA_{#1,#2}(#3)}}
\def\rat{f}
\def\pol{f}
\def\unipol{f_c}

\newcommand{\Ch}{{\operatorname{Ch}}}
\newcommand{\Elim}{{\operatorname{Elim}}}
\newcommand{\proj}{{\operatorname{proj}}}
\newcommand{\h}{{\operatorname{h}}}

\newcommand{\hh}{\mathrm{h}}
\newcommand{\aff}{\mathrm{aff}}
\newcommand{\Spec}{{\operatorname{Spec}}}
\newcommand{\Res}{{\operatorname{Res}}}

\numberwithin{equation}{section}
\numberwithin{theorem}{section}

\def\house#1{{%
    \setbox0=\hbox{$#1$}
    \vrule height \dimexpr\ht0+1.4pt width .5pt depth \dimexpr\dp0+.8pt\relax
    \vrule height \dimexpr\ht0+1.4pt width \dimexpr\wd0+2pt depth \dimexpr-\ht0-1pt\relax
    \llap{$#1$\kern1pt}
    \vrule height \dimexpr\ht0+1.4pt width .5pt depth \dimexpr\dp0+.8pt\relax}}

\newcommand{\Address}{{
\bigskip
\footnotesize
\textsc{Centro di Ricerca Matematica Ennio De Giorgi, Scuola Normale Superiore, Pisa, 56126, Italy}\par\nopagebreak
\textit{E-mail address:} \texttt{marley.young@sns.it}
}}


\title[]
{$S$-integral preperiodic points for monomial semigroups over number fields}

\begin{abstract}
We consider semigroup dynamical systems defined by several monnomials over a number field $K$. We prove a finiteness result for preperiodic points of such systems which are $S$-integral with respect to a non-preperiodic point $\beta$, which is uniform as $\beta$ varies over number fields of bounded degree. This generalises results of Baker, Ih and Rumely, which were made uniform by Yap, and verifies a special case of a natural generalisation of a conjecture of Ih.
\end{abstract}

\author {Marley Young}

\maketitle


\section{Introduction}

Let $K$ be a number field with algebraic closure $\overline K$, let $S$ be a finite set of places of $K$, containing all the archimedean ones, and let $\alpha,\beta \in \overline K$. We say that $\alpha$ is $S$-integral relative to $\beta$ if no conjugate of $\alpha$ meets any conjugate of $\beta$ at primes lying outside of $S$. 

Given a rational map $f: \P^1 \to \P^1$ of degree at least 2 defined over $K$, let $f^n$ denote the $n$-fold composition given by
$$
f^n = \underbrace{f \circ \cdots \circ f}_{n \text{ times}}.
$$
A point $\alpha \in \P^1$ is \emph{preperiodic} for $f$ if $f^{n+m}(\alpha) = f^m(\alpha)$ for some $n \geq 1$ and $m \geq 0$. By analogy with well-known results in Diophantine geometry, S. Ih has conjectured the following finiteness property for $S$-integral preperiodic points.

\begin{conjecture}[Ih]
\label{conj:Ih}
Let $f:\P^1 \to \P^1$ be a rational map of degree at least 2, defined over $K$, and $\beta \in \P^1(\overline K)$ a point which is not preperiodic for $f$. Then there are only finitely many preperiodic points $\alpha \in \P^1(\overline K)$ for $f$ which are $S$-integral relative to $\beta$.
\end{conjecture}

Baker-Ih-Rumely \cite{BIR} have proved this conjecture in the cases where $f(z)=z^d$, $d \geq 2$ is a power map or $f$ is a Latt\'{e}s map associated to an elliptic curve $E/K$. Uniform results (as $\beta$ varies over number fields of bounded degree) in these cases were shown by Yap \cite{Yap}. Ih and Tucker \cite{IT} have also proved the case where $f$ is a Chebyshev map. For a general rational map $f$, Conjecture~\ref{conj:Ih} remains open, but was proved by Petsche \cite{P} under certain local conditions on the non-preperiodic point $\beta$. In this latter setting, the author has proved some effective bounds \cite{Young}.

Let $f_i : \P^1 \to \P^1$ for $i \in I$ be rational maps defined over $K$, and let $\cG := \langle f_i \rangle_{i \in I}$ denote the semigroup generated by the $f_i$ under composition. That is, $\cG$ consists of all maps of the form
\begin{equation}
\label{eq:fk}
f_{i_1 \cdots i_n} := f_{i_n} \circ \cdots \circ f_{i_1}, \quad n \geq 1, \: i_1,\ldots
,i_n \in I.
\end{equation}
By convention, we take the above expression with $n=0$ to denote the identity map. Given $\alpha \in \overline K$, write $\cO_\cG(\alpha) = \{ \alpha \} \cup \{ g(\alpha) : g \in \cG \}$ for the \emph{forward orbit} of $\alpha$ under $\cG$. When $\cG = \langle f \rangle$ is generated by a single map, $\cO_\cG(\alpha) = \cO_f(\alpha)$ is just the classical forward orbit of $\alpha$ under $f$. When $\cG = \langle f_1, f_2 \rangle$ is generated by two maps, $\cO_\cG(\alpha)$ can be seen as the binary tree represented below.

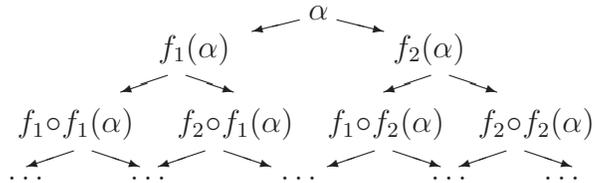
\begin{figure}[h]
\setlength{\unitlength}{1.25cm}
\vskip 2pt 
\begin{picture}(170,1.6)(9,1.2)
\put(15.1,2.9){$\alpha$}
\put(15.0,2.9){\vector(-4,-1){0.5}}
\put(15.4,2.9){\vector(4,-1){0.5}}
\put(13.5,2.5){$f_1(\alpha)$}  
\put(16.0,2.5){$f_2(\alpha)$}  
\put(13.6,2.3){\vector(-3,-1){0.5}}
\put(13.8,2.3){\vector(3,-1){0.5}}
\put(16.4,2.3){\vector(-3,-1){0.5}}
\put(16.6,2.3){\vector(3,-1){0.5}}
\put(12.0,1.7){$f_1\comp f_1(\alpha)$}  
\put(13.7,1.7){$f_2 \comp f_1(\alpha)$}  
\put(15.3,1.7){$f_1 \comp f_2(\alpha)$}  
\put(16.9,1.7){$f_2\comp f_2(\alpha)$}  
\put(12.8,1.5){\vector(3,-1){0.5}}
\put(12.6,1.5){\vector(-3,-1){0.5}}
\put(14.2,1.5){\vector(3,-1){0.5}}
\put(14,1.5){\vector(-3,-1){0.5}}
\put(15.8,1.5){\vector(-3,-1){0.5}}
\put(16.0,1.5){\vector(3,-1){0.5}}
\put(17.2,1.5){\vector(-3,-1){0.5}}
\put(17.4,1.5){\vector(3,-1){0.5}}
\put(11.9,1.2){\ldots}
\put(13.2,1.2){\ldots}
\put(14.8,1.2){\ldots}
\put(16.4,1.2){\ldots}
\put(17.6,1.2){\ldots}
\end{picture}
\end{figure}

The study of semigroup dynamical systems goes back to work of Silverman~\cite{Si93} who investigated the finiteness of $S$-integers in orbits generated by a finite set of rational functions, and more recently appears in several works addressing important questions in arithmetic dynamics  such as the Mordell-Lang Conjecture for endomorphisms of semiabelian varieties~\cite{GTZ1}, the orbit intersection problem~\cite{GN,GTZ}, see also~\cite{YZ}, the  irreducibility of compositions of polynomials~\cite{FGS,GMS,HBM}.

We are interested in formulating a semigroup version of Conjecture~\ref{conj:Ih}. To this end, we must define an appropriate notion of preperiodic points for a semigroup dynamical system. Several points of view can be adapted in this regard, as there are multiple ways to generalise the classical notion of preperiodicity under a single map. Firstly (see for example \cite[\textsection 1.3]{Kaw}), one can consider points $\alpha \in \overline K$ such that $\cO_\cG(\alpha)$ is finite. Let us call such points \emph{strongly preperiodic}. Then we have the following conjecture.

\begin{conjecture} \label{conj:StrongSemigroupIh}
Let $f_i : \P^1 \to \P^1$, $i \in I$ be rational maps of degree at least 2, defined over $K$, $\cG := \langle f_i \rangle_{i \in I}$, and let $\beta \in \P^1(\overline K)$ be a point which is not strongly preperiodic for $\cG$. Then there are only finitely many strongly preperiodic points $\alpha \in \P^1(\overline K)$ for $\cG$ which are $S$-integral relative to $\beta$.
\end{conjecture}

The problem with this definition of preperiodic in our context, is that it is stronger than being preperiodic for each individual element of the semigroup, and hence Conjecture~\ref{conj:StrongSemigroupIh} is equivalent to Conjecture~\ref{conj:Ih}. It can in fact be quickly resolved when the semigroup $\cG$ contains any dynamically unrelated elements.

\begin{prop}
If $\cG = \langle f_i \rangle_{i \in I}$ has two elements $f$ and $g$ with distinct Julia sets (when viewed as complex dynamical systems), then $\cG$ has only finitely many strongly preperiodic points, and so Conjecture~\ref{conj:StrongSemigroupIh} holds vacuously.
\end{prop}


\begin{proof}
If $\cG$ has elements $f$ and $g$ with distinct Julia sets, then by \cite[Corollary~1.3]{BD}, they have only finitely many preperiodic points in common, and hence $\cG$ has only finitely many strongly preperiodic points.
\end{proof}

We say that a point $\alpha \in \overline K$ is \emph{preperiodic} for $\cG$, if
$f_{i_1 \cdots i_n}(\alpha) = f_{i_1 \cdots i_m}(\alpha)$ for some $n > m \geq 0$ and $i_1,\ldots,i_n \in I$. It is clear from this definition that a semigroup $\cG$ can have preperiodic points which are not preperiodic for any of its generators, and so this context should hopefully lead to a more interesting generalisation of Conjecture~\ref{conj:Ih}. Note that there is yet another, even weaker notion of preperiodicity (see for example \cite[(1.7)]{OY}) which we will not discuss here. We conjecture the following integrality result for finitely generated semigroups of rational functions.

\begin{conjecture}
\label{conj:SemigroupIh}
Let $f_1,\ldots,f_s : \P^1 \to \P^1$ be rational maps of degree at least 2, defined over $K$, $\cG := \langle f_1,\ldots,f_s \rangle$, and let $\beta \in \P^1(\overline K)$ be a point which is not preperiodic for $\cG$. Then there are only finitely many preperiodic points $\alpha \in \P^1(\overline K)$ for $\cG$ which are $S$ integral relative to $\beta$.
\end{conjecture}

The main result of this paper is the special case of Conjecture~\ref{conj:SemigroupIh} in which the semigroup $\cG$ is generated by monomials. This generalises \cite[Theorem~2.1]{BIR} (the case of a single power map), and is uniform in the sense of \cite[Theorem~1.2]{Yap}.

\begin{theorem}
\label{thm:main}
Let $\cG=\langle f_1,\ldots,f_s \rangle$, where each $f_i$ is a monomial of the form $f_i(z) = a_i z^{d_i}$, $a_i \in K \setminus \{0 \}$, $d_i \in \Z$ with $|d_i| \geq 2$, and let $D \geq 1$. Then there exists an effectively computable constant $C>0$, depending only on $\cG$, $S$ and $[K:\Q]$, such that for any $\beta \in \overline K$ with $[K(\beta):K] \leq D$ which is not preperiodic for $\cG$, there are at most $C$ preperiodic points for $\cG$ in $\overline K$ which are $S$-integral relative to $\beta$.
\end{theorem}

Theorem~\ref{thm:main}, together with the results of \cite{Young} (and the extension of some of these to the semigroup setting) form a large part of the author's PhD thesis \cite{YThesis}.

The proof of Theorem~\ref{thm:main} follows in the same vein as \cite[Theorem~2.1]{BIR}, with the uniformity obtained in a similar manner to \cite{Yap}. However, there are a number of novel obstacles and technical subtleties. Firstly, it is necessary to establish a quantitative equidistribution result for points of small height for dynamical sequences obtained from the given semigroup, which is uniform in the sense that it depends only on the semigroup, and not the choice of sequence. The semigroup case also requires more general application of lower bounds for linear forms in logarithms, including in the non-archimedean setting (which could be avoided in the case of a single map). Moreover, in order for our bounds coming from linear forms in logarithms to not be vacuous, we require lower bounds on the degree of a preperiodic point in terms of its (pre)period. This is easily achieved in the case of a single power map, since in $\zeta_n$ is a primitive $n$-th root of unity, then $[\Q(\zeta_n):\Q] = \varphi(N) \ll N^{1/2}$, where $\varphi$ is the Euler totient function. However, the theory of general binomial factorisation is more complicated than one would think, and hence we need some additional arguments in our setting.

Note also that Conjecture~\ref{conj:SemigroupIh} holds for semigroups generated by (even infinitely many) Chebyshev maps, but that this is immediate from the case of a single Chebyshev map, proved in \cite{IT} (see \cite[Theorem~1.14]{YThesis}).

The paper is structured as follows. In Section~\ref{sec:prelim}, we give some background on canonical heights associated to sequences of rational functions, lower bounds on linear forms in logarithms, adelic measures and potential theory on the Berkovich projective line. In Section~\ref{sec:QuantEquid} we establish some properties of adelic measures associated to sequences of rational maps obtained from a rational semigroup, which imply the aforementioned uniform quantitative equidistribution of points of small height for such sequences. In section~\ref{sec:preperStruct}, we use results on the factorisation of lacunary polynomials to prove a useful characterisation of preperiodic points for monomial semigroups. We use this characterisation together with bounds on linear forms in logarithms to give a lower bound on the distance from a non-preperiodic point to a preperiodic point of such semigroups in Section~\ref{sec:DistBds}. We then prove Theorem~\ref{thm:main} in Section~\ref{sec:ProofMain}.

\subsection*{Acknowledgements} The author is very grateful to Robert Benedetto and Holly Krieger for useful comments on the paper.

\section{Preliminaries} \label{sec:prelim}

Let $K$ be a number field, and let $M_K$ denote the set of all places of $K$. For each $v \in M_K$, let $| \cdot |_v$ denote absolute value normalised as follows. If $v$ is archimedean then
$$
|x|_v = \begin{cases} |x|^{\frac{1}{[K:\Q]}}, & K_v \cong \R, \\ |x|^{\frac{2}{[K:\Q]}}, & K_v \cong \C. \end{cases}
$$
If $v$ is non-archimedean and corresponds to a prime ideal $\p$ of the ring of integers $\cO_K$ of $K$, then $|x|_v = (N_K \p)^{-\frac{\ord_\p(x)}{[K:\Q]}}$, where $N_K \p = |\cO_K/\p|$ is the norm of $\p$ and $\ord_\p(x)$ denotes the exponent of $\p$ in the prime ideal decomposition of $x$, with $\ord_\p(0)=\infty$. 

Then we have the \emph{product formula}
$$
\prod_{v \in M_K} |x|_v = 1
$$
for $0 \neq x \in K$. There is a unique extension of $| \cdot |_v$ to $\overline K_v$, also denoted $| \cdot |_v$. Given a finite extension $L/K$, some $\beta \in L$ and a place $v$ of $K$, we have the \emph{extension formula}
\begin{equation} \label{eq:ExtFormula}
\sum_{\sigma : L/K \hookrightarrow \overline K_v} \log |\sigma(\beta)|_v = \sum_{w \mid v} \log |\beta|_w.
\end{equation}

\subsection{Reduction of rational maps} \label{subsec:Reduction}

Given a non-archimedean place $v$ of $K$, denote the ring of integers and residue field of $K$ with respect to $v$ by $\cO_v$ and $k_v$ respectively. Given a rational function $\rat: \P^1 \to \P^1$ of degree $d$ defined over $K$, represented as
$$
\rat([x,y]) = [g(x,y),h(x,y)],
$$
where $g,h \in \cO_v[x,y]$ are homogeneous polynomials of degree $d$ with no common irreducible factors in $K[x,y]$ and at least one coefficient of $g$ or $h$ has $v$-adic absolute value 1, we say that $\rat$ has \emph{good reduction} at $v$ if the reductions of $f$ and $g$ modulo $v$ are relatively prime in $k_v[x,y]$. Otherwise we say that $\rat$ has \emph{bad reduction} at $v$. Note that if $f_1,f_2: \P_1 \to \P^1$ are rational functions defined over $K$, each with good reduction, then $f_1 \circ f_2$ also has good reduction \cite[Proposition~4.8]{B}.

\subsection{Heights of algebraic numbers} \label{subsec:Heights} Recall that the \emph{(absolute logarithmic) height} of an algebraic number $\alpha$ is given by
$$
h(\alpha) = \sum_{v \in M_K} \log^+ |\alpha|_v,
$$
where $K$ is any number field containing $\alpha$, and $\log^+x = \log \max \{ 1, x \}$. The height is independent of the choice of number field, invariant under Galois conjugation, and by the famous theorem of Northcott, there are only finitely algebraic numbers of bounded degree and height. We have, for any algebraic numbers $\alpha_1, \ldots, \alpha_r$
\begin{equation} \label{eq:heightIneq}
h(\alpha_1 \cdots \alpha_r) \leq h(\alpha_1) + \cdots + h(\alpha_r).
\end{equation}
Moreover, for any algebraic number $\alpha \neq 0$ contained in a number field $K$, and any $\lambda \in \Q$,
\begin{equation} \label{eq:heightPower}
h(\alpha^\lambda) = |\lambda| h(\alpha),
\end{equation}
and for any subset $U$ of $M_K$
\begin{equation} \label{eq:fundIneq}
-h(\alpha) \leq \sum_{v \in U} \log |\alpha|_v \leq h(\alpha)
\end{equation}
(see for example \cite[\textsection \textsection 1.5]{BG}).


Let $\bm{f} = (f_i)_{i=1}^\infty$ be a sequence of rational maps $\P^1 \to \P^1$ defined over $K$, of respective degrees $d_i \geq 2$. Then $c(f_i) := \sup_{x \in \P^1(\overline K)} \left| \frac{1}{d_i} h(f_i(x))-h(x) \right| < \infty$. We say that a sequence $\bm{f}$ is \emph{bounded} if $c(\bm{f}) := \sup_{i \geq 1} c(f_i) < \infty$. Let $T$ be the shift map which sends $\bm{f} = (f_i)_{i=1}^\infty$ to $T(\bm{f}) := (f_{i+1})_{i=1}^\infty$. Kawaguchi \cite[Theroem~A]{K} constructed a canonical height function associated to a bounded sequence.

\begin{prop}
\label{prop:SeqCanHeight}
There is a unique way to attach to each bounded sequence $\bm{f} = (f_i)_{i=1}^\infty$ a height function $\hat  h_{\bm{f}} : \P^1(\overline K) \to \R$ such that
\begin{enumerate}
\item[(a)] $\sup_{x \in \P^1(\overline K)} \left| \hat h_{\bm{f}}(x) - h(x) \right| \leq 2c(\bm{f})$;
\item[(b)] $\hat h_{T(\bm{f})} \circ f_1 = d_1 \hat h_{\bm{f}}$.
\end{enumerate}
Moreover, $\hat h_{\bm{f}}$ is non-negative, and $\hat h_{\bm{f}}(x)=0$ if and only if $x$ is $\bm{f}$-preperiodic, in the sense that the forward orbit $O_{\bm{f}}^+(x) := \{ x,f_1(x),f_2(f_1(x)),\ldots \}$ is finite.

We call $\hat h_{\bm{f}}$ a \emph{canonical height function} for $\bm{f}$.
\end{prop}


\defconstant{SHLB}

Similarly, we say that a rational semigroup $\cG = \langle f_i \rangle_{i \in I}$ generated by functions of degree $d_i \geq 2$ defined over $K$, is \emph{bounded} if $c(\cG) := \sup_{i \geq I} c(f_i) < \infty$. We say that a sequence $\bm{f} = (f_{i_k})_{k=1}^\infty$ is \emph{obtained from $\cG$} if $i_k \in I$ for all $k \geq 1$. Note the following
\begin{lemma} \label{lem:LowBdSeqCanHeight}
Let $\cG = \langle f_i \rangle_{i \in I}$ be a bounded rational semigroup defined over $K$ with $d := \sup_{i \in I} \deg f_i < \infty$, and suppose $\beta \in \overline K$ is not preperiodic for $\cG$. Then there is a constant $\coSHLB > 0$, depending only on $\cG$, $K$ and $[K(\beta):K]$ such that
$$
\hat h_{\bm{f}}(\beta) \geq \coSHLB
$$
for any sequence $\bm{f}$ obtained from $\cG$.
\end{lemma}

\begin{proof}
If $\hat h_{\bm{f}}(\beta) = \varepsilon$ for a sequence $\bm{f}=(f_{i_k})_{k=1}^\infty$ obtained from $\cG$, then by Proposition~\ref{prop:SeqCanHeight}
\begin{align*}
h(f_{i_n} \circ \cdots \circ f_{i_1}(\beta)) & \leq \hat h_{T^n(\bm{f})}(f_{i_n} \circ \cdots \circ f_{i_1}(\beta)) + 2c(\cG) \\
& = d_{i_1} \cdots d_{i_n} \hat h_{\bm{f}}(\beta) + 2c(\cG) \leq d^n \varepsilon + 2c(\cG)
\end{align*}
for any $n \geq 1$. In particular, if $\varepsilon$ can be chosen arbitrarily small, this produces arbitrarily many points of bounded degree and height, a contradiction.
\end{proof}

\subsection{Lower bounds on linear forms in logarithms} \label{subsec:linearforms}

Bounds on linear forms in logarithms are crucial to our proof of Theorem~\ref{thm:main}. These were first developed by Baker (see for example \cite[Chapters~2~and~3]{Ba}), and have wide application to transcendence theory and Diophantine problems. We use the following (see \cite[Proposition~3.10]{BEG}), which combines results of Matveev \cite{Mat} and Yu \cite{Yu} in the archimedean and non-archimedean settings respectively.

\begin{theorem} \label{thm:LinearFormsInLogs}
Let $\alpha_1,\ldots,\alpha_n \in K \setminus \{0 \}$, and $b_1,\ldots,b_n \in \Z$, not all zero. Put
\begin{align*}
\Lambda & := \alpha_1^{b_1} \cdots \alpha_n^{b_n} -1 \\
\Theta &:= \prod_{i=1}^n \max \left( h(\alpha_i), \frac{2}{[K:\Q] \left( \log (3[K:\Q]) \right)^3} \right), \\
B & := \max(3,|b_1|,\ldots,|b_n|)
\end{align*}
Let $v$ be a place of $K$, and write
$$
N(v) := \begin{cases} 2 & \text{if $v$ is infinite} \\
N_K \p & \text{if $v=\p$ is finite}. \end{cases}
$$
Suppose that $\Lambda \neq 0$. Then for $v \in M_K$ we have
$$
\log |\Lambda|_v > - c_1(n,[K:\Q]) \frac{N(v)}{\log N(v)} \Theta \log B,
$$
where $c_1(n,d) = 12d(16ed)^{3n+2}\max(1,\log d)^2$, with  (note an extra factor of $d$ due to our different normalisation of absolute values).
\end{theorem}

\subsection{The Berkovich projective line and adelic measures} \label{subsec:adelmeas}

For each place $v$ of $K$, the \emph{Berkovich projective line} $\BP(\C_v)$ is a Hausdorff, uniquely path connected, compact topological space which contains $\P^1(\C_v)$ as a dense subspace. When $v$ is archimedean, we simply have $\BP(\C_v) = \P^1(\C_v)$, but in the non-archimedean setting $\BP(\C_v)$ contains many additional points, corresponding to disks (or nested sequences thereof) in $\C_v$. For details of the construction of $\BP(\C_v)$, and its basic properties, see \cite[\textsection 3]{FRL}. One can study potential theory on $\BP(\C_v)$, and appropriate notions of a Laplacian $\Delta$, a mutual energy $( \cdot , \cdot )_v$ for signed Radon measures on $\BP(\C_v)$ etc. are introduced for example in \cite[\textsection 4]{FRL}

When $v$ is an infinite place of $K$, let $\lambda_v$ denote the probability measure proportional to Lebesgue measure on the unit circle $S^1 \subset \P^1(\C_v)$. When $v$ is finite, let $\lambda_v$ denote the point mass at the Gauss point (i.e. the point corresponding to the disk $D(0,1) \subset \C_v$) in $\BP(\C_v)$.

An \emph{adelic measure} $\rho = \{ \rho_v \}_{v \in M_K}$ consists of a probability measure $\rho_v$ on $\BP(\C_v)$ with continuous potentials for each place $v \in M_K$, such that $\rho_v = \lambda_v$ for all but finitely many places $v$. To an adelic measure $\rho$ we can associate an \emph{adelic height}, given by
\begin{equation} \label{eq:adelHeight}
h_\rho(F) := \frac{1}{2} \sum_{v \in M_K} ([F] - \rho_v, [F]-\rho_v)_v,
\end{equation}
for any finite set $F \subset \overline K$ invariant under the action of $\Gal(\overline K/K)$ (here $[F]$ is the measure $|F|^{-1} \sum_{z \in F} \delta_z$, where $\delta_z$ is the point mass at $z$, and we note that due to our normalisation of absolute values, $( \cdot, \cdot)_v$ coincides with the $\llrrparen{ \cdot , \cdot }_v$ used by Favre and Rivera-Letelier). For $\alpha \in \P^1(\overline K)$, we let $h_\rho(\alpha) = h_\rho(F)$ where $F$ is the orbit of $\alpha$ under the action of $\Gal(\overline K/K)$. When $\rho = \{ \lambda_v \}_{v \in M_K}$, we have $h_\rho = h$, the usual height. Given $0 < \kappa \leq 1$, we say that an adelic measure $\rho$ has \emph{$\kappa$-H\"{o}lder continuous potentials} if for each place $v$, there exists a function $g_v : \BP(\C_v) \to \R$ with $\Delta g_v = \rho_v - \lambda_v$ and a constant $C_v > 0$ such that 
\begin{equation} \label{eq:HoldContPot}
|g_v(z)-g_v(w)| \leq \begin{cases} C_v|z-w|_v^\kappa \text{ for } z,w \in \C_v & \text{if $v$ is archimedean}, \\ C_v \ud(z,w)^\kappa \text{ for } z,w \in \BP(\C_v) & \text{otherwise}. \end{cases}
\end{equation}
Here $\ud$ is the metric defined in \cite[\textsection \textsection 4.7]{FRL}. Note that in the archimedean case, we are using the standard Euclidean distance instead of the spherical distance (used in Favre and Rivera-Letelier's definition), since it is more convenient for our purposes.

If $v$ is an archimedean place, let $H^1(\P^1(\C_v))$ denote the Sobolev space of functions $\P^1(\C_v) \to \R$ whose weak partial derivatives exist and are locally square integrable. Given $f \in H^1(\P^1(\C_v))$, we define 
$$
\langle f,f \rangle = \int df \wedge d^c f = \frac{1}{2\pi} \int_{\C_v} \left( \frac{\partial f}{\partial x} \right)^2 + \left( \frac{\partial f}{\partial y} \right)^2  dx \wedge dy.
$$ There is also an analogous definition in the non-archimedean setting, see \cite[\textsection \textsection 5.5]{FRL}.

\section{Equilibrium measures for bounded sequences and quantitative equidistribution} \label{sec:QuantEquid}

Favre and Rivera-Letelier \cite[Theorem~7]{FRL} proved the following quantitative equidistribution result.

\defconstant{FRL}

\begin{theorem} \label{thm:QuantEquid}
Suppose $\rho = \{ \rho_v \}_{v \in M_K}$ is an adelic measure with $\kappa$-H\"{o}lder continuous potentials $g_v$, with constants $C_v >0$. Let $V = \{ v \in M_K : \rho_v \neq \lambda_v \}$. There exists a constant $\coFRL > 0$, depending only on $|V|$, $\kappa$, and $C_v$ for $v \in V$, such that for any place $v \in M_K$, any $\cC^1$ test function $\phi$ on $\BP(\C_v)$, and any finite $\Gal(\overline K/K)$-invariant subset $F$ of $\overline K$ we have
\begin{equation*}
\left| \frac{1}{|F|} \sum_{\alpha \in F} \phi(\alpha) - \int_{\BP(\C_v)} \phi \, d \rho_v \right| \leq \left( 2 h_{\rho}(F) + \coFRL \frac{\log |F|}{|F|} \right)^{\frac{1}{2}} \langle \phi, \phi \rangle^{\frac{1}{2}} +  \frac{\Lip(\phi)}{|F|^{1/\kappa}}.
\end{equation*}
\end{theorem}

Here $\Lip(\phi)$ denotes the Lipschitz constant of $\phi$ with respect to the distances as in \eqref{eq:HoldContPot}. Our assumption that $F$ does not contain $\infty$ allows us to replace the spherical distance with the Euclidean distance at archimedean places, as noted in \cite[Proposition~2.2]{Yap} (see \cite[Theorem~3.15]{YThesis} for an explicit computation of $\coFRL$). In fact, it is possible to loosen the $\cC^1$ regularity requirement on the test function $\phi$ to H\"{o}lder continuity at archimedean places. This form of quantitative equidistribution can be used to bound the number of conjugates of a point of zero height (with respect to an adelic measure) lying in a small disc. For the following see \cite[Proposition~2.3]{Yap} or \cite[Proposition~3.16]{YThesis}.

\defconstant{EDB}

\begin{prop} \label{prop:EquidDiscBd}
Let $\cZ$ be a finite $\Gal(\overline K/K)$-invariant subset of $\overline K$, and let $\rho = \{ \rho_v \}_{v \in M_K}$ be an adelic measure with H\"{o}lder continuous potentials (of exponent $\kappa$) such that $h_\rho(\cZ)=0$. Let $v$ be an archimedean place of $K$, then for any $w \in \C_v$ and $0 < \varepsilon < 1/e$, we have
$$
| \cZ \cap D(w,\varepsilon)| \leq  \rho_v(D(w,e \varepsilon))|\cZ|  + \coEDB \left( \frac{1}{\varepsilon |\cZ|^{1/\kappa-1}} + \sqrt{|\cZ| \log |\cZ|} \right),
$$
where $\coEDB>0$ is a constant, effectively computable in terms of the constants and exponent of H\"{o}lder continuous potentials of $\rho$.
\end{prop}


Favre and Rivera-Letelier also show \cite[Theorem~8]{FRL} that for a rational function $\rat$, $\rho_{\rat} = \{ \rho_{\rat,v} \}_{v \in M_K}$ is an adelic measure with H\"{o}lder continuous potentials, where $\rho_{\rat,v}$ is the equilibrium measure of $\rat$ at $v$, and that the adelic height $h_{\rho_{\rat}}$ coincides with the canonical height $\canheight{\rat}$ for $\rat$. We will show that a similar statement holds for sequences $\bm{f}$ of rational functions obtained from a finitely generated rational semigroup $\cG$, and in particular that the constants and exponent of H\"{o}lder continuity depend only on $\cG$.

\begin{prop}
\label{prop:seqMeas}
Let $\cG = \langle f_1, \ldots, f_s \rangle$ be a semigroup of rational functions defined over $K$ with $\deg f_i = d_i \geq 2$. Let $\bm{f} = (f_{i_k})_{k=1}^\infty$ be a sequence obtained from $\cG$, that is, each $i_k \in \{ 1, \ldots, s \}$. Then
\begin{enumerate}
\item[(a)] For each $v \in M_K$, the sequence $(d_{i_1} \cdots d_{i_n})^{-1} f_{i_1 \cdots i_n}^* \lambda_v$ converges to a measure $\rho_{\bm{f},v}$;
\item[(b)] $\rho_{\bm{f}} := \{ \rho_{\bm{f},v} \}_{v \in M_K}$ is an adelic measure, whose adelic height $h_{\rho_{\bm{f}}}$ coincides with the canonical height $\hat h_{\bm{f}}$ associated to $\bm{f}$;
\item[(c)] $\rho_{\bm{f}}$ has $\kappa$-H\"{o}lder continuous potentials with constants $C_v >0$, where $\kappa$ and $C_v$ depend only on $\cG$.
\end{enumerate}
\end{prop}

\begin{proof}
Let $v$ be any place of $K$. As in \cite[Proposition~6.5]{FRL}, for each $i \in \{ 1, \ldots , s \}$, we can find a Lipschitz potential $g_i$ such that $\Delta g_i = d_i^{-1} f_i^* \lambda_v - \lambda_v$. By iterating successively, we see that for all $n \geq 1$,
$$
(d_{i_1} \cdots d_{i_n})^{-1} f_{i_1 \cdots i_n}^* \lambda_v - \lambda_v = \Delta \left( \sum_{k=1}^n (d_{i_1} \cdots d_{i_{k-1}})^{-1} g_{i_k} \circ f_{i_1 \cdots i_{k-1}} \right).
$$
The series $\Delta \left( \sum_{k \geq 1} (d_{i_1} \cdots d_{i_{k-1}})^{-1} g_{i_k} \circ f_{i_1 \cdots i_{k-1}} \right) =: \Delta G_v$ converges uniformly on $\BP(\C_v)$, and so the sequence of probability measures $\{ (d_{i_1} \cdots d_{i_n})^{-1} f_{i_1 \cdots i_n}^* \lambda_v \}_{n \geq 1}$ converges to a measure which we denote $\rho_{\bm{f},v} = \lambda_v + \Delta G_v$. This proves (a).

If $v$ is such that $f_i$ has good reduction at $v$ for all $i \in \{1,\ldots,s\}$, we have $d_i^{-1} f_i^* \lambda_v = \lambda_v$ for all $i \in \{1, \ldots, s \}$, and so it is easy to see that $\rho_{\bm{f},v}=\lambda_v$. This is the case for all but finitely many places, so $\rho_{\bm{f}} := \{ \rho_{\bm{f},v} \}_{v \in M_K}$ is an adelic measure. Thus, by the uniqueness statement in Proposition~\ref{prop:SeqCanHeight}, to complete the proof of (b), it suffices to show that $d_{i_1} h_{\rho_{\bm{f}}} = h_{\rho_{T(\bm{f})}} \circ f_{i_1}$. We proceed as in the proof of \cite[Theorem~8]{FRL}, by introducing a \emph{cross ratio} as follows. For $v \in M_K$ and $z_0,z_1,w_0,w_1 \in \C_v$, set
$$
(z_0,z_1,w_0,w_1)_v = \log \frac{|z_0-w_0|_v \cdot |z_1-w_1|_v}{|z_0-w_1|_v \cdot |z_1-w_0|_v}.
$$
This extends to a continuous function on the set
$$
\cD := \{(z_0,z_1,w_0,w_1) \in \BP(\C_v)^4 : \text{if $z_i = w_j$ then } z_i \notin \P^1(\C_v) \}.
$$
Given signed Radon measures $\mu_0, \mu_1, \nu_0, \nu_1$ on $\BP(\C_v)$ such that $(\cdot, \cdot, \cdot, \cdot)_v$ is integrable with respect to $\mu_0 \otimes \mu_1 \otimes \nu_0 \otimes \nu_1$ on $\cD$, let
$$
(\mu_0,\mu_1,\nu_0,\nu_1)_v := \int_{\cD} (z_0,z_1,w_0,w_1)_v d\mu_0(z_0) \otimes d \mu_1(z_1) \otimes d \nu_0(w_0) \otimes d \nu_1(w_1).
$$
Then for any rational function $f$ of degree $d \geq 1$ defined over $K$, and measures signed Radon measures $\mu_0,\mu_1,\nu_0,\nu_1$ on $\BP(\C_v)$ such that
$$
(f_* \mu_0, f_* \mu_1,\nu_0,\nu_1)_v \: \text{ and } \: (\mu_0, \mu_1, f^* \nu_0, f^* \nu_1 )_v
$$
are well-defined, we have \cite[Formule~de~transformation]{FRL}
\begin{equation} \label{eq:TransFormula}
(f_* \mu_0, f_* \mu_1,\nu_0,\nu_1)_v = d^{-1} \cdot (\mu_0, \mu_1, f^* \nu_0, f^* \nu_1 )_v.
\end{equation}
Moreover, if $\rho = \{ \rho_v \}$ is an adelic measure and $F,F'$ and $F''$ are finite subsets of $\P^1(\overline K)$ which are invariant under the action of $\Gal(\overline K/K)$, then \cite[Lemme~6.2]{FRL}
\begin{itemize}
\item If $F''$ is disjoint from $F \cup F'$, then
\begin{equation} \label{eq:CR1}
\sum_{v \in M_K} ( [F],[F'],[F''], \rho_v )_v = h_\rho(F)-h_\rho(F').
\end{equation}
\item If $F$ and $F'$ are disjoint, then
\begin{equation} \label{eq:CR2}
\sum_{v \in M_K} ( \rho_v, [F], \rho_v, [F'])_v = h_\rho(F)+h_\rho(F').
\end{equation}
\end{itemize}
Now, let $F \subset \P^1(\overline K)$ be a finite $\Gal(\overline K/K)$-invariant set. Then $f_{i_1}(F)$ and $f_{i_1}^{-1}(F)$ are finite and $\Gal(\overline K/K)$-invariant. Also, $(f_{i_1})_*[F] = [f_{i_1}(F)]$, and if $F$ contains no critical values of $f_{i_1}$, then any element of $F$ has exactly $d_{i_1}$ preimages under $f_{i_1}$, and $d_{i_1}^{-1} f_{i_1}^*[F] = [f_{i_1}^{-1}(F)]$.

Given finite $\Gal(\overline K/K)$-invariant sets $F,F' \subset \P^1(\overline K)$, let $F''$ be a finite $\Gal(\overline K/K)$-invariant subset of $\P^1(\overline K)$ containing no critical values of $f_{i_1}$, and such that $f_{i_1}^{-1}(F'')$ is disjoint from $F \cup F'$. Then
$$
( (f_{i_1})_*[F], (f_{i_1})_*[F'], [F''], \rho_{T(\bm{f}),v} )_v = ( [f_{i_1}(F)], [f_{i_1}(F')], [F''], \rho_{T(\bm{f}),v} )_v
$$
and
$$
([F],[F'], d_{i_1}^{-1} f_{i_1}^*[F''], \rho_{\bm{f},v} )_v = ([F],[F'], [f_{i_1}^{-1}(F'')], \rho_{\bm{f},v} )_v
$$
are well-defined and \eqref{eq:TransFormula} gives
$$
([f_{i_1}(F)],[f_{i_1}(F')],[F''],\rho_{T(\bm{f}),v})_v = d_{i_1} ([F],[F'],[f_{i_1}^{-1}(F'')], \rho_{\bm{f},v})_v.
$$
Summing over all places, \eqref{eq:CR1} gives
$$
h_{\rho_{T(\bm{f})}}(f_{i_1}(F)) - h_{\rho_{T(\bm{f})}}(f_{i_1}(F')) = d_{i_1} \left( h_{\rho_{\bm{f}}}(f_{i_1}(F)) - h_{\rho_{\bm{f}}}(f_{i_1}(F)) \right),
$$
which implies that the function $h_{\rho_{T(\bm{f})}} \circ f_{i_1} - d_{i_1} \cdot h_{\rho_{\bm{f}}}$ is constant. A similar argument with \eqref{eq:CR2} shows that if $F'$ contains no critical values of $f_{i_1}$, and $f_{i_1}^{-1}(F')$ is disjoint from $F$, then
$$
h_{\rho_{T(\bm{f})}}(f_{i_1}(F)) + h_{\rho_{T(\bm{f})}}(f_{i_1}(F')) = d_{i_1} \left( h_{\rho_{\bm{f}}}(f_{i_1}(F)) + h_{\rho_{\bm{f}}}(f_{i_1}(F)) \right),
$$
and so $h_{\rho_{T(\bm{f})}} \circ f_{i_1} - d_{i_1} \cdot h_{\rho_{\bm{f}}}$ is identically zero, completing the proof of (b).

Suppose now that $v$ is either an infinite place, or a finite place of bad reduction. Let $N \geq 0$ be an integer, and set $d = \min_{1 \leq i \leq s} d_i \geq 2$. We have for any $z,w \in \BP(\C_v)$,
\begin{align*}
|G_v(z) & -G_v(w)|_v \leq \sum_{k=1}^N (d_{i_1} \cdots d_{i_{k-1}})^{-1} | g_{i_k} \circ f_{i_1 \cdots i_{k-1}}(z) - g_{i_k} \circ f_{i_1 \cdots i_{k-1}}(w)|_v \\
& \qquad \qquad + \max_{1 \leq i \leq s} \sup |g_i|_v \sum_{k=N}^\infty (d_{i_1} \cdots d_{i_k})^{-1} \\
& \leq \left( \sum_{k=0}^{N-1} d^{-k} \max_{1 \leq i \leq s} \Lip_v(g_i) \left( \max_{1 \leq i \leq s} \Lip(f_i) \right)^k \right) \ud(z,w) + \max_{1 \leq i \leq s} \sup |g_i|_v \sum_{k=N}^\infty d^{-k}.
\end{align*}
Then, writing $E_v = \max_{1 \leq i \leq s} \Lip_v(g_i)$, $L_v = \max \left \{ 2d \max_{1 \leq i \leq s} \Lip_v(f_i) \right \}$ and $B_v = \max_{1 \leq i \leq s} \sup |g_i|_v$, we have
\begin{align*}
|G_v(z)-G_v(w)|_v & \leq \frac{E}{\frac{L}{d}-1} \left( \frac{L}{d} \right)^N \ud(z,w) + dB \left( \frac{1}{d} \right)^N \\
& = \left( \frac{E}{\frac{L}{d}-1} L^N \ud(z,w) + dB \right) \left( \frac{1}{d} \right)^N.
\end{align*}
If we take $N=\lfloor-\log(\ud(z,w))/\log(L) \rfloor$, then $N = -\log_L \ud(z,w) - \delta$ for some $0 \leq \delta < 1$, and we obtain
\begin{align} \label{eq:HolderBound}
|G_v(z)-G_v(w)| & \leq \left( \frac{E}{\frac{L}{d}-1} L^{-\log_L \ud(z,w)} \ud(z,w) L^{-\delta} + dB \right) d^{\frac{\log \ud(z,w)}{\log L} + \delta} \notag \\
& = d^{\delta} \left( \frac{dE}{L-d} L^{-\delta} + dB \right) \ud(z,w)^{\frac{\log d}{\log L}} \notag \\ 
& < \left(\frac{2d^2 E}{L}+d^2B \right) \ud(z,w)^{\frac{\log d}{\log L}},
\end{align}
noting that $L \geq 2d$. This completes the proof of (c).
\end{proof}

We define below a useful archimedean test function to which we will apply quantitative equidistribution, and control its regularity.

\begin{lemma} \label{lem:ArchTruncTest}
Let $v$ be an archimedean place of $K$, and let $0 < \delta < R$. Let $\phi_{\delta,R} : \P^1(\C_v) \to \R$ be the function given by $\phi_{\delta,M}(z) = \log \min(R,\max(\delta,|z|_v))$. Then $\phi_{\delta,R} \in H^1(\P^1(\C_v))$ is Lipschitz on $\C_v$ with constant $1/\delta$, and satisfies
$$
\langle \phi_{\delta,R}, \phi_{\delta,R} \rangle = \log R - \log \delta.
$$
\end{lemma}

\begin{proof}
It is clear that in local coordinates $x,y$, $\phi_{\delta,R}$ has weak partial derivatives
$$
\frac{\partial \phi_{\delta,R}}{\partial x}(x,y) = \I \left \{ \delta \leq \sqrt{x^2+y^2} \leq R \right \} \frac{x}{x^2+y^2},
$$
and similarly with respect to $y$, where $\I$ is an indicator function. Thus $\phi_{\delta,R} \in H^1(\P^1(\C_v))$, with
\begin{align*}
\langle \phi_{\delta,R}, \phi_{\delta,R} \rangle & = \frac{1}{2\pi} \int_{\C_v} \left( \frac{\partial \phi_{\delta,R}}{\partial x} \right)^2 + \left( \frac{\partial \phi_{\delta,R}}{\partial x} \right)^2  dx dy = \frac{1}{2\pi} \int_{\delta \leq \sqrt{x^2+y^2} \leq R} \frac{dx dy}{x^2+y^2} \\ 
& = \frac{1}{2\pi} \int_0^{2 \pi} \int_\delta^R \frac{dr d \theta}{r} = \log R - \log \delta.
\end{align*}
Also, $\phi_{\delta,R}$ is Lipschitz with constant $1/\delta$ since this is the maximum of the derivative of $\log x$ on the interval $[\delta,R]$.
\end{proof}

\section{Structure of preperiodic points for monomial semigroups} \label{sec:preperStruct}

Given a semigroup $\cG$ generated by monomials, it is easy to see that all of its preperiodic points are roots of some binomial $X^N-a$. In the case of a single power map, we have $a=1$, and so the irreducible factors of the binomial are easily understood, namely, they are the cyclotomic polynomials $\Phi_d$ of degree $\varphi(d)$ for $d \mid N$. In generality, however, it is much more difficult to control the structure and degree of irreducible factors of $X^N-a$, since its some of its obvious factors may still be reducible. For example,
\begin{align*}
X^6+27 &=  (X^2+3)(X^4-3X^2+9) \\
& = (X^2+3)(X^2-3X+3)(X^2+3X+3).
\end{align*}
Fortunately, due to results of Schinzel \cite{Sch}, the extent to which the non-binomial parts of the ``obvious'' factorisation of $X^N-a$ can be reduced further is limited in a specific sense. Moreover, in our context, the constant term $a$ is a product of powers of finitely many elements of the number field $K$. We show that this implies that any preperiodic point for $\cG$ can be written in the following convenient presentation, from which we later extract a useful lower bound for its distance to a non-preperiodic point using Theorem~\ref{thm:LinearFormsInLogs}.

\begin{prop} \label{prop:PreperStruct}
Let $\cG = \langle f_1, \ldots, f_s \rangle$, where for $i=1,\ldots,s$, $f_i(z)=a_i z^{d_i}$ with $a_i \in K \setminus \{ 0 \}$ and $|d_i| \geq 2$, and let $\alpha \in \overline K \setminus \{0,\infty\}$ be a preperiodic point for $\cG$. Then $\alpha = \zeta \gamma$, where $\zeta$ is a root of unity, and either $\gamma = 1$ or $\gamma$ is a root of
\begin{equation} \label{eq:gammastruct}
X^M - a_1^{m_1} \cdots a_s^{m_s} b,
\end{equation}
for some $M \geq 1$, $m_1,\ldots,m_s \in \Z$ with $|m_1|,\ldots,|m_s| < M$, and
\begin{align*}
b \in \cB := \big\{ b & \in K : b \text{ is a root of } X^L - \xi a_1^{\ell_1} \cdots a_s^{\ell_s} \text{ for some } L \geq 1, \, \ell_i \in \Z \\
& \text{with } |\ell_i| < L \text{ for } 1 \leq i \leq s, \text{ and } \xi \in K \text{ a root of unity} \big\}.
\end{align*}
For all $b \in \cB$,
\begin{equation} \label{eq:htbd}
h(b) \leq h(a_1) + \cdots + h(a_s),
\end{equation}
and so $\cB$ is a finite set. Moreover, $M/2 \leq [K(\gamma):K] \leq M$, and
$$
[K(\zeta \gamma):K] \geq \max \left( \frac{M}{w(K)}, \frac{\max ([K(\zeta):K], M/2)}{\min([K(\zeta):K],M)} \right),
$$
where $w(K)$ denotes the number of roots of unity in $K$. In particular, if $\zeta$ is a primitive $Q$-th root of unity, then
\begin{equation} \label{eq:PreperDegBd}
\log [K(\alpha): K] \gg \log (MQ),
\end{equation}
where the implied constant depends only on $[K:\Q]$.
\end{prop}

\begin{proof}
By definition, we have $f_{i_1 \cdots i_n} (\alpha) = f_{i_1 \cdots i_m}(\alpha)$ for some $n > m \geq 0$ and $i_1,\ldots,i_n \in \{ 1, \ldots, s \}$. Now
$$
f_{i_1 \cdots i_n} (z) - f_{i_1 \cdots i_m}(z) = a_{i_n} a_{i_{n-1}}^{d_{i_n}} \cdots a_{i_1}^{d_{i_2} \cdots d_{i_n}} z^{d_{i_1} \cdots d_{i_n}} - a_{i_m} a_{i_{m-1}}^{d_{i_m}} \cdots a_{i_1}^{d_{i_2} \cdots d_{i_m}} z^{d_{i_1} \cdots d_{i_m}},
$$
and so, since $\alpha$ is non-zero, it is a root of
$$
X^N - a_1^{k_1} \cdots a_s^{k_s},
$$
where $N = d_{i_1} \cdots d_{i_n} - d_{i_1} \cdots d_{i_m}$, and $k_1,\ldots,k_s$ are integers with $|k_i| \leq N$. Write $a = a_1^{k_1} \cdots a_s^{k_s}$, and following the notation of \cite{Sch}, let $\zeta_\ell$ denote a primitive $\ell$-th root of unity, and let
$$
e(a,K) = \begin{cases} 0 & \text{if } a \text{ is a root of unity}, \\
\max \{ \ell : a = \xi x^\ell \text{ for some root of unity } \xi \text{ and } x \in K \}, & \text{otherwise}, \end{cases}
$$
and
$$
E(a,K) = \begin{cases} 0 & \text{if } a \text{ is a root of unity}, \\
\max \{ \ell : a = y^k \text{ for some } y \in K(\zeta_\ell) \}, & \text{otherwise}. \end{cases}
$$
If $a$ is a root of unity, then so is $\alpha$, and we are done. Otherwise, let $\ell := e(a,K) \geq 1$, so that $a = \xi x^\ell$ for some root of unity $\xi$ and some $x \in K$. Note that $\xi = a/x^\ell \in K$. Let $u = \gcd(N,\ell)$, $M = N/u$, and consider the binomial
$$
G(X) := X^M - x^{\frac{\ell}{u}} = X^M - (\xi^{-1} a)^{\frac{1}{u}} \in K[X].
$$
Recall that
$$
a = a_1^{k_1} \cdots a_s^{k_s} = \left( a_1^{m_1} \cdots a_s^{m_s} \right)^u a_1^{\ell_1} \cdots a_s^{\ell_s},
$$
where $k_i = m_i u + \ell_i$, so $|m_i| \leq M$ and $|\ell_i| < u$. Then
$$
G(X):=X^M-(\xi^{-1} a)^{\frac{1}{u}}=X^M-a_1^{m_1} \cdots a_s^{m_s} b \in K[X]
$$
for some $b \in \cB$ (namely a $K$-rational root of $X^u - a_1^{\ell_1} \cdots a_s^{\ell_s}$), and is hence of the desired form \eqref{eq:gammastruct}. We note that the set $\cB \subset K$ is finite, since it is a set of bounded height. Indeed, from \eqref{eq:heightIneq} and \eqref{eq:heightPower}, for any $b = (a_1^{\ell_1} \cdots a_s^{\ell_s})^{1/L} \in \cB$ we have
$$
h(b) \leq \frac{\ell_1}{L} h(a_1) + \cdots + \frac{\ell_s}{L} h(a_s) \leq h(a_1) + \cdots + h(a_s).
$$
Now, from a classical theorem of Capelli \cite[Theorem~21]{Sch3}, $G$ is reducible over $K$ if and only if
\begin{enumerate}
\item[(a)] for some prime divisor $p$ of $M$, $x^{\ell/u} = y^p$ for some $y \in K$; or
\item[(b)] 4 divides $M$ and $x^{\ell/u} = -4y^4$ for some $y \in K$.
\end{enumerate}
By \cite[Lemma~1]{Sch2}, $e(x^{\ell/u},K)=\ell/u$, which is coprime to $M$ by definition, and so $G$ is not reducible according to (a). By a theorem of Schinzel \cite{Sch4}, if $G$ is reducible according to (b), the two resulting factors of degree $M/2$ are irreducible. Thus, if $\gamma$ is a root of $G$, we have $[K(\gamma):K] \geq M/2$. Furthermore,
$$
\gamma^N = \gamma^{Mu} = \xi^{-1} a,
$$
and so, recalling that $\alpha^N=a$, we must have $\gamma = \zeta \alpha$ for some root of unity $\zeta$. By \cite[Lemma~3]{Sch}, each irreducible factor of $X^N-a$ is a polynomial in $X^{N/\gcd(N,E(a,K))}$, and so
$$
[K(\alpha):K] \geq \frac{N}{\gcd(N,E(a,K))}.
$$
But by \cite[Lemma~2]{Sch}
$$
E(a,K) \mid e(a,K) w(K),
$$
and so
$$
[K(\alpha):K] \geq \frac{N}{\gcd(N,E(a,k))} \geq \frac{N}{w(K) \gcd(N,e(a,K))} = \frac{M}{w(K)}.
$$
To conclude, consider the tower of extensions 
$$
K \subset K(\alpha)=K(\zeta \gamma) \subset K(\zeta,\gamma).
$$
By the tower law we have
$$
[K(\alpha):K] = \frac{[K(\gamma,\zeta):K]}{[K(\gamma,\zeta):K(\alpha)]} \geq \frac{\max( [K(\gamma):K], [K(\zeta):K] )}{[K(\gamma,\zeta):K(\gamma \zeta)]},
$$
but
$$
[K(\gamma,\zeta):K(\gamma \zeta)] \leq \min ( [K(\gamma):K], [K(\zeta):K)] ),
$$
and so
$$
[K(\alpha):K] \geq \max \left( \frac{M}{w(K)}, \frac{\max ([K(\zeta):K], M/2)}{\min([K(\zeta):K],M)} \right).
$$
If $\zeta=\zeta_Q$ is a primitive $Q$-th root of unity, we have $[K(\zeta):K] = \varphi(Q)/[K:\Q]$. Thus, if $\varphi(Q) > [K:\Q]M^2$, then
$$
[K(\alpha) : K] \gg \sqrt{\frac{\varphi(Q)}{[K:\Q]}} \gg Q^{\frac{1}{4}} \gg (MQ)^{\frac{1}{6}},
$$
and otherwise $Q \ll M^4$, so (since $w(K)$ is clearly bounded in terms of $[K:\Q]$),
$$
[K(\alpha):K] \gg M \gg (MQ)^{\frac{1}{5}}.
$$
In either case, $\log \max(MQ) \ll \log [K(\alpha):K]$, as desired.
\end{proof}

We also note that at each place, the size of a non-zero preperiodic point for a monomial semigroup $\cG$ is bounded from above and below depending only on $\cG$.

\begin{lemma} \label{lem:MonomPrePerSizeBds}
Let $\cG = \langle f_1, \ldots, f_s \rangle$, where for $i=1,\ldots,s$, $f_i(z)=a_i z^{d_i}$ with $a_i \in K \setminus \{ 0 \}$ and $|d_i| \geq 2$, and let $\alpha \in \overline K \setminus \{0,\infty\}$ be a preperiodic point for $\cG$. Let $v$ be any place of $K$. Then $r \leq |\alpha|_v \leq 1/r$, where
$$
r = r(\cG) := \min_{1 \leq i \leq s} \min \left( |a_i|_v^{\frac{1}{|d_i|-1}}, |a_i|_v^{-\frac{1}{|d_i|-1}} \right) \leq 1.
$$
\end{lemma}

\begin{proof}
Suppose $0 \neq |\alpha|_v < r$. Then for any $i \in \{ 1 ,\ldots ,s \}$, if $d_i > 0$, we have
$$
|f_i(\alpha)|_v = |a_i|_v |\alpha|_v^{|d_i|} < |\alpha|_v < r.
$$
On the other hand, if $d_i < 0$, we have
$$
|f_i(\alpha)|_v = |a_i|_v |\alpha|_v^{-|d_i|} > \frac{1}{|\alpha|_v} > \frac{1}{r}.
$$
Iterating this argument, for any sequence $(i_k)_{k=1}^\infty$ with $i_k \in \{1,\ldots,s\}$, the sequence
$$
\min \left( |f_{i_1 \cdots i_n}(\alpha)|_v, \frac{1}{|f_{i_1 \cdots i_n}(\alpha)|_v} \right),
$$
is strictly decreasing, and so $\alpha$ is not preperiodic for $\cG$.
\end{proof}

\section{Bounds on distance to nearest preperiodic point} \label{sec:DistBds}

We are now able to apply the lower bounds on linear forms in logarithms, stated in \textsection \ref{sec:prelim}, to prove the following. This will allow us to obtain results which are uniform as $\beta$ varies over a number field.

\defconstant{DB}

\begin{theorem} \label{thm:DistBd}
Let $\cG = \langle f_1, \ldots, f_s \rangle$, where for $i=1,\ldots,s$, $f_i(z)=a_i z^{d_i}$ with $a_i \in K \setminus \{ 0 \}$ and $|d_i| \geq 2$, and let $\beta \in K$ be a non-preperiodic point for $\cG$. Let $\alpha \in \overline K \setminus \{0\}$ be a preperiodic point for $\cG$, and let $v$ be a place of $K$. Then there exists a constant $\coDB > 0$, depending only on $\cG$, $[K:\Q]$ and the place of $\Q$ below $v$, such that
$$
\min_{\sigma : K(\alpha)/K \hookrightarrow \overline K_v} \log |\sigma(\alpha)-\beta|_v > -\coDB(h(\beta)+1) \log [K(\alpha):K].
$$
\end{theorem}

\begin{proof}

Let $\sigma : K(\alpha)/K \hookrightarrow \overline K_v$ be any embedding, and write $\sigma(\alpha) = \zeta_{Q_\sigma} \gamma$ as in Proposition~\ref{prop:PreperStruct}. Let $Q:=\max_\sigma Q_\sigma$, set 
$$
x= \begin{cases} \zeta_{Q_\sigma} \gamma \beta^{-1} & |\sigma(\alpha)|_v \leq |\beta|_v \\ (\zeta_{Q_\sigma} \gamma)^{-1} \beta & |\sigma(\alpha)|_v > |\beta|_v, \end{cases}
$$ 
and note that $|x|_v \leq 1$. Then
$$
|x-1|_v = \frac{|x^{MQ_\sigma}-1|_v}{|x^{MQ_\sigma-1} + \cdots + x + 1|_v} \geq \frac{|x^{MQ_\sigma}-1|_v}{MQ_\sigma},
$$
since $|x^{MQ_\sigma-1} + \cdots + 1|_v \leq MQ_\sigma$ by the triangle inequality. Now, by \eqref{eq:gammastruct},
$$
x^{MQ_\sigma} = \left( \gamma^{MQ_\sigma} \beta^{-MQ_\sigma} \right)^{\pm 1} = \left( a_1^{Q_\sigma m_1} \cdots a_s^{Q_\sigma m_s} b^{Q_\sigma} \beta^{-MQ_\sigma} \right)^{\pm 1},
$$
so applying Theorem~\ref{thm:LinearFormsInLogs}, we have
\begin{equation} \label{eq:eq1}
\log |x^{MQ_\sigma}-1|_v > -c_1(s+2,[K:\Q]) \frac{N(v)}{\log N(v)} \Theta \log \max(3,MQ).
\end{equation}
Here
$$
\Theta \ll h(\beta)+1,
$$
where the implied constant depends only on $\cG$ and $[K:\Q]$, since by \eqref{eq:htbd}, $h(b) \ll_\cG 1$. Note also that if $v$ is non-archimedean, $N(v) \leq p^{[K:\Q]}$, where $p$ is the rational prime below $v$. Hence, using \eqref{eq:PreperDegBd} and \eqref{eq:fundIneq}, there exists a constant $\coDB > 0$, depending only on $[K:\Q]$, $\cG$ and the place of $\Q$ below $v$, such that 
\begin{align*}
\log |\sigma(\alpha)-\beta|_v & \geq \log |\beta|_v + \log |x-1|_v \\
& \geq -\left( h(\beta) + \log(MQ_\sigma) \right) + \log|x^{MQ_\sigma}-1|_v \\
& \geq - \coDB (h(\beta)+1) \log [K(\alpha):K],
\end{align*}
as required.
\end{proof}

\section{Proof of Theorem~\ref{thm:main}} \label{sec:ProofMain}

By replacing $K$ with $K(\beta)$, and $S$ with the set of places $S_{K(\beta)}$ lying over $S$, we are reduced to proving the theorem with $\beta \in K$, noting that $[K(\beta):\Q] \leq D[K:\Q]$ and $|S_{K(\beta)}| \leq D|S|$. Recall that $\beta \in K$ is not preperiodic for the semigroup of monomials 
$$
\cG = \langle f_1,\ldots,f_s \rangle = \langle a_1 z^{d_1},\ldots, a_s z^{d_s} \rangle,
$$ 
$0 \neq a_i \in K$, $|d_i| \geq 2$ for $i=1,\ldots,s$, and suppose $\alpha \in \overline K$ is a non-zero preperiodic point for $\cG$, which is $S$-integral relative to $\beta$. Then $f_{i_1 \cdots i_n} (\alpha) = f_{i_1 \cdots i_m}(\alpha)$ for some $n > m \geq 0$ and $i_1,\ldots,i_n \in \{ 1, \ldots, s \}$. Thus $\alpha$ is $\bm{f}$-preperiodic, where $\bm{f} = ( f_{i_1},\ldots, f_{i_m}, \overline{ f_{i_{m+1}},\ldots, f_{i_n} })$ is obtained from $\cG$. Let $g_1 = f_{i_1 \cdots i_m}$ and $g_2 = f_{i_{m+1} \cdots i_n}$, and note that $g_1$ and $g_2$ are monomials, say $g_1(z) = az^k$ and $g_2(z)=bz^\ell$, where $a,b \in K \setminus \{ 0 \}$, are products of the $a_i$, $i \in \{1,\ldots,s\}$, and $k,\ell \in \Z$ are products of powers of the $d_i$. Since $\alpha \neq 0$, it is a root of
\begin{equation} \label{eq:alpha}
\frac{g_2(g_1(z))-g_1(z)}{g_1(z)} = a^{\ell-1} b z^{k(\ell-1)} - 1 = 0.
\end{equation}
Note then that
\begin{align}  \label{eq:heightComp1}
\hat h_{\bm{f}}(\beta) & = \lim_{n \to \infty} \frac{1}{|k| \ell^{2n}} \sum_{v \in M_K} \log^+ \left| g_2^{2n}(g_1(z)) \right|_v \notag \\
& = \lim_{n \to \infty} \frac{1}{|k| \ell^{2n}} \sum_{v \in M_K} \log^+ \left| a^{\ell^n} b^{\frac{\ell^{2n}-1}{\ell-1}} \beta^{k\ell^{2n}} \right|_v \notag \\
& = \sum_{v \in M_K} \max \left( 0, \frac{1}{|k|} \log |a|_v + \frac{1}{|k|(\ell-1)} \log |b|_v + \log |\beta|_v \right)
\end{align}
Hence, using  \eqref{eq:alpha} and the product formula, we get
\begin{align} \label{eq:heightComp}
\hat h_{\bm{f}}(\beta) & = \Bigg( \sum_{v \in M_K} \max \left( -\frac{1}{k} \log |a|_v - \frac{1}{k(\ell-1)} \log |b|_v, \log |\beta|_v \right) \notag \\
& \qquad \qquad \qquad \pm \frac{1}{k} \sum_{v \in M_K} \log |a|_v \pm \frac{1}{k(\ell-1)} \sum_{v \in M_K} \log |b|_v \Bigg) \notag \\
& = \sum_{v \in M_K} \log \max( |\alpha|_v, |\beta|_v).
\end{align}
We will evaluate the sum
\begin{equation} \label{eq:Gamma}
\Gamma = \frac{1}{[K(\alpha):K]} \sum_{v \in M_K} \sum_{\sigma: K(\alpha)/K \hookrightarrow \overline K_v} \log |\sigma(\alpha) - \beta|_v
\end{equation}
in two different ways. Firstly, an application of the product formula shows that $\Gamma = 0$. On the other hand, we will use the integrality hypothesis and quantitative equidistribution together with Theorem~\ref{thm:DistBd} to show that $\Gamma$ is very close to $\hat h_{\bm{f}}(\beta)$ if $[K(\alpha):K]$ is sufficiently large (independent of the particular sequence $\bm{f}$). This bounds the degree of preperiodic $\alpha$ which are $S$-integral relative to $\beta$, and consequently proves the finiteness of such points (as their height is bounded).

Firstly, using the extension formula \eqref{eq:ExtFormula} gives
\begin{align*}
\Gamma & = \frac{1}{[K(\alpha):K]} \sum_{v \in M_K} \sum_{\sigma: K(\alpha)/K \hookrightarrow \overline K_v} \log |\sigma(\alpha)-\beta|_v \\
& = \frac{1}{[K(\alpha):K]}\sum_{w \in M_{K(\alpha)}} \log |\alpha-\beta|_w. 
\end{align*}
Since $\beta$ is not preperiodic for $\cG$, the product formula gives $\Gamma = 0$.

For any place $v$ of $K$, by abuse of notation we will also denote by $\alpha$ its image in $\overline K_v$ under some fixed embedding. Note that for any other embedding $\sigma: K(\alpha)/K \hookrightarrow \overline K_v$, we have $\sigma(\alpha) = \eta \alpha$ for some root of unity $\eta$, and so $|\sigma(\alpha)|_v = |\eta \alpha|_v = |\alpha|_v$. Suppose $v \notin S$. If $|\beta|_v > |\alpha|_v$, then by the ultrametric inequality, for each $\sigma: K(\alpha)/K \hookrightarrow \overline K_v$, we have $|\sigma(\alpha)-\beta|_v = |\beta|_v$. On the other hand, if $|\beta|_v \leq |\alpha|_v$, we have $|\sigma(\alpha)-\beta|_v \leq |\alpha|_v$, and we split into two cases. If $|\alpha|_v \leq 1$, then the integrality hypothesis gives $|\sigma(\alpha)-\beta|_v \geq 1 \geq |\alpha|_v$, and so $|\sigma(\alpha)-\beta|_v = |\alpha|_v$. If instead $|\alpha|_v > 1$, then the integrality hypothesis gives $|\beta|_v \leq 1$, and so the ultrametric inequality again gives $|\sigma(\alpha)-\beta|_v = |\alpha|_v$. It follows that for each $v \notin S$,
\begin{equation*}
\frac{1}{[K(\alpha):K]} \sum_{\sigma: K(\alpha)/K \hookrightarrow \overline K_v} \log |\sigma(\alpha)-\beta|_v = \log \max( |\alpha|_v, |\beta|_v ),
\end{equation*}
and so, since $S$ is finite, we can exchange the limit and sum in \eqref{eq:Gamma} to obtain
\begin{align} \label{eq:Gamma2}
0 = \Gamma & =\sum_{v \in S} \frac{1}{[K(\alpha):K]} \sum_{\sigma: K(\alpha)/K \hookrightarrow \overline K_v} \log |\sigma(\alpha) - \beta|_v + \sum_{v \notin S} \log \max( |\alpha|_v, |\beta|_v ).
\end{align}

We will now show that for each $v \in S$, when $[K(\alpha):K]$ is sufficiently large,
\begin{align} \label{eq:squah}
& \left| \frac{1}{[K(\alpha):K]} \sum_{\sigma: K(\alpha)/K \hookrightarrow \overline K_v} \log |\sigma(\alpha) -\beta|_v - \log \max (|\alpha|_v, |\beta|_v ) \right| \notag \\
& \qquad \qquad \qquad \qquad  < \frac{\coSHLB+\max(0,h(\beta)-2c(\cG))}{2|S|},
\end{align}
where $\coSHLB > 0$ is the constant (depending only on $\cG$ and $[K:\Q]$) in Lemma~\ref{lem:LowBdSeqCanHeight}, and we recall
$$
c(\cG) = \max_{1 \leq i \leq s} c(f_i) = \max_{1 \leq i \leq s} \sup_{x \in \P^1(\overline K)} \left| \frac{1}{|d_i|} h(f_i(x)) - h(x) \right| < \infty.
$$
Inserting this in \eqref{eq:Gamma2} and using \eqref{eq:heightComp} gives $\hat h_{\bm{f}}(\beta)< \max \left( \coSHLB, h(\beta) - 2 c(\cG) \right)$, a contradiction by Proposition~\ref{prop:SeqCanHeight}~(a).

\subsection*{Non-archimedean case} 

For each non-archimedean $v \in S$, the inequality \eqref{eq:squah} is trivial if $|\beta|_v > |\alpha|_v$ or $|\beta|_v < |\alpha|_v$. Hence we can assume that $|\beta|_v = |\alpha|_v$.

\begin{lemma} \label{lem:nonArchBds}
Let $v$ be non-archimedean, and suppose $|\beta|_v = |\alpha|_v$. Then for each $0 < \varepsilon < 1$, the number $U(\varepsilon)$ of roots of unity $\zeta \in \overline K_v$ with $|\zeta \alpha - \beta|_v < \varepsilon |\beta|_v$ is bounded only in terms of $K$, $v$, and $\varepsilon$.
\end{lemma}

\begin{proof}
Note that if $\zeta$ and $\zeta'$ are roots of unity with $|\zeta \alpha - \beta|_v < \varepsilon |\beta|_v$ and $|\zeta' \alpha - \beta|_v < \varepsilon |\beta|_v$, then we have $|\zeta-\zeta'|_v < \varepsilon |\beta|_v/|\alpha|_v = \varepsilon$, and so $\zeta'' := \zeta^{-1} \zeta'$ is a root of unity with $|1-\zeta''|_v < \varepsilon$. There are only finitely many such $\zeta''$, since if $p$ is the rational prime under $v$, the only roots of unity $\xi \in \overline K_v$ with $|1-\xi|_v < 1$ are ones with order $p^n$ for some $n$. If $\xi$ is a primitive $p^n$-th root of unity, then $|1-\xi|_v = p^{-[K_v:\Q_p]/p^{n-1}(p-1)}$ so $1 > \varepsilon > |1-\xi|_v$ for only finitely many $n$.
\end{proof}

Fix $0 < \varepsilon < 1$. For each embedding $\sigma: K(\alpha)/K \hookrightarrow \overline K_v$, we have $|\sigma(\alpha) - \beta|_v \leq |\beta|_v$, so by Theorem~\ref{thm:DistBd},
\begin{align*}
\log |\beta|_v & \geq \frac{1}{[K(\alpha):K]} \sum_{\sigma: K(\alpha)/K \hookrightarrow \overline K_v} \log |\sigma(\alpha) - \beta|_v \\
& \geq \frac{1}{[K(\alpha):K]} \big( \left( [K(\alpha):K] - U(\varepsilon) \right) \log(\varepsilon|\beta|_v) \\
& \qquad \qquad \qquad \qquad - U(\varepsilon) \coDB ( h(\beta) + 1) \log [K(\alpha):K] \big).
\end{align*}
Subtracting $\log |\beta|_v$ from both sides of the above inequality gives
\begin{align*}
&\left| \frac{1}{[K(\alpha):K]} \sum_{\sigma:K(\alpha)/K \hookrightarrow \overline K_v} \log |\sigma(\alpha) -\beta|_v - \log |\beta|_v \right| \\
& \leq \log |\beta|_v-\frac{([K(\alpha):K]-U(\varepsilon)) \log(\varepsilon |\beta|_v) - U(\varepsilon) \coDB ( h(\beta) + 1) \log [K(\alpha):K]}{[K(\alpha):K]} \\
& =  -\log \varepsilon + \frac{U(\varepsilon) \left( \log \varepsilon +\log |\beta|_v + (h(\beta)+1) \coDB \log [K(\alpha):K] \right)}{[K(\alpha):K]} \\
& \qquad \qquad \qquad \leq -\log \varepsilon + \frac{U(\varepsilon) \left( \log \varepsilon +(h(\beta)+1) \left(1+ \coDB \log [K(\alpha):K] \right) \right)}{[K(\alpha):K]},
\end{align*}
where the last inequality follows from \eqref{eq:fundIneq}. Since $\varepsilon$ was arbitrary, taking $[K(\alpha):K]$ sufficiently large (independently of $\bm{f}$ and $\beta$) establishes \eqref{eq:squah} in this case.

\subsection*{Archimedean case} Now suppose $v$ is archimedean. Note that given a monomial $f(z)=az^d$, if $\delta_{z,\varepsilon}$ denotes Lebesgue measure on the circle of radius $\varepsilon$ about $z$, then we have for any $R>0$
\begin{align*}
\frac{1}{|d|} f^* \delta_{0,R} & = \frac{1}{|d|} f^* \Delta \log \max(R, |z|) \\
& = \Delta \log \max \left( \left( \frac{R}{|a|_v} \right)^{\frac{1}{d}}, |z| \right) = \delta_{0,(R/|a|_v)^{1/d}}.
\end{align*}
Hence, with $g_1,g_2$ defined as above the equation \eqref{eq:alpha}, we have from Proposition~\ref{prop:seqMeas}~(a),
$$
\rho_{\bm{f},v} = \lim_{n \to \infty} \frac{1}{|k\ell^n|} (g_2^n \circ g_1)^* \lambda_v  = \delta_{0,|a|_v^{-1/k} |b|_v^{-1/(k(\ell-1))}}
$$
(following the calculation in \eqref{eq:heightComp1}), but $|a|_v^{-\frac{1}{k}} |b|_v^{-\frac{1}{k(\ell-1)}} = |\alpha|_v$ by \eqref{eq:alpha}. Thus, by Jensen's formula (see \cite{Co}, p. 280) applied to $f(z)=z-\beta$,
\begin{equation} \label{eq:Jensen}
\int_{\P^1(\C_v)} \log |z-\beta|_v d \rho_{\bm{f},v} = \log \max( |\alpha|_v, |\beta|_v).
\end{equation}
Let $\cZ$ be the set of $\Gal(\overline K/K)$-conjugates of $\alpha$, so $|\cZ| = [K(\alpha):K]$. Let
$$
\varepsilon := \min \left( \frac{1}{e}, \frac{r}{e[K(\alpha):K]^{1/2}} \right),
$$
where $r=r(\cG)$ is defined in Lemma~\ref{lem:MonomPrePerSizeBds}. From \eqref{eq:Jensen} we can write
\begin{align} \label{eq:TermSplit}
& \frac{1}{[K(\alpha):K]} \sum_{\sigma : K(\alpha)/K \hookrightarrow \C} \log |\sigma(\alpha)-\beta|_v - \log \max ( |\alpha|_v, |\beta|_v )  \notag \\
& \qquad = \frac{1}{|\cZ|} \left( \sum_{\substack{z \in \cZ \\ |z-\beta| < \varepsilon}} \log |z-\beta|_v - \left| \cZ \cap D(\beta,\varepsilon) \right| \log \varepsilon \right) \notag \\
& \qquad \quad + \frac{1}{|\cZ|} \sum_{z \in \cZ} \log \max(\varepsilon, |z-\beta|_v) - \int_{\P^1(\C_v)} \log \max(\varepsilon, |z-\beta|_v ) d \rho_{\bm{f},v}  \notag \\
& \qquad \quad + \int_{\P^1(\C_v)} \left( \log \max (\varepsilon, |z-\beta|_v) - \log |z-\beta|_v \right) d \rho_{\bm{f},v}.
\end{align}
By Theorem~\ref{thm:DistBd}, we have
\begin{align*}
\left| \sum_{\substack{z \in \cZ \\ |z-\beta| < \varepsilon}} \log |z-\beta|_v - \left| \cZ \cap D(\beta,\varepsilon) \right| \log \varepsilon \right| & \ll \left(h(\beta) \log |\cZ| - \log \varepsilon \right) \left| \cZ \cap D(\beta,\varepsilon) \right|,
\end{align*}
and by Proposition~\ref{prop:EquidDiscBd},
$$
|\cZ \cap D(\beta,\varepsilon)| \ll \rho_{\bm{f},v}(D(\beta,e \varepsilon))|\cZ| + \frac{1}{\varepsilon} + \sqrt{|\cZ| \log |\cZ|}.
$$
Recalling the $\rho_{\bm{f},v}$ is normalised Haar measure on the circle $|z|_v = |\alpha|_v$ and $e \varepsilon < r \leq |\alpha|_v$, we have that $\rho_{\bm{f},v}(D(\beta,e\varepsilon))$ is maximised in the case where $|\beta|_v = |\alpha|_v$, where it is given by twice the angle subtended by a chord of length $e \varepsilon$ on a circle of radius $|\alpha|_v$. That is,
$$
\rho_{\bm{f},v}(D(\beta,e \varepsilon)) \leq 2 \sin^{-1} \left( \frac{e \varepsilon}{2 |\alpha|_v} \right) \leq 2 \sin^{-1} \left( \frac{e \varepsilon}{2 r} \right) \ll \varepsilon.
$$
Putting this together gives
\begin{align} \label{eq:term1}
\frac{1}{|\cZ|} & \left| \sum_{\substack{z \in \cZ \\ |z-\beta| < \varepsilon}} \log |z-\beta|_v - \left| \cZ \cap D(\beta,\varepsilon) \right| \log \varepsilon \right| \notag \\
& \qquad \ll \left( h(\beta)\log |\cZ| - \log \varepsilon \right) \left( \varepsilon + \frac{1}{\varepsilon |\cZ|} + \sqrt{\frac{\log |\cZ|}{|\cZ|}} \right) \ll \frac{h(\beta)}{|\cZ|^{1/3}}.
\end{align}

Now, let $\phi(z) := \phi_{\varepsilon, \frac{1}{r}+|\beta|_v}(z-\beta)$, as defined in Lemma~\ref{lem:ArchTruncTest}. Noting that $\phi(z)=\log \max(\varepsilon, |z-\beta|_v)$ for $z \in \cZ$, we can apply Proposition~\ref{thm:QuantEquid} to $\phi$ to obtain
\begin{align} \label{eq:term2}
& \left| \frac{1}{|\cZ|} \sum_{z \in \cZ} \log \max(\varepsilon,|z-\beta|_v) - \int_{\P^1(\C_v)} \log \max( \varepsilon, |z-\beta|_v) d \rho_{\bm{f},v} \right| \notag \\
& \qquad \quad \ll \left( \frac{(\log |\beta|_v - \log \varepsilon) \log |\cZ|}{|\cZ|} \right)^{\frac{1}{2}} + \frac{1}{\varepsilon |Z|} \ll \frac{h(\beta)}{|\cZ|^{1/3}}.
\end{align}

Finally, arguing similarly to above, the last term in \eqref{eq:TermSplit} is maximised when $|\alpha|_v = |\beta|_v$, in which case
$$
\left| \int_{\P^1(\C_v)} \left( \log \max (\varepsilon, |z-\beta|_v) - \log |z-\beta|_v \right) d \rho_{\bm{f},v} \right| \leq 2 \left| \int_0^{\theta_0} \log (|\alpha|_v \theta) d \theta \right|,
$$
where $\theta_0$ is the angle subtended by a chord of length $\varepsilon$ on a circle of radius $|\alpha|_v$. Recalling again that $|\alpha|_v \geq r$, we have $\theta_0 \ll \varepsilon$, and so
$$
\left| \int_0^{\theta_0} \log (|\alpha|_v \theta) d \theta \right| = \left | \theta_0 \log (|\alpha|_v \theta_0) - \theta_0 \right| \ll \varepsilon \log \varepsilon \ll \frac{1}{|\cZ|^{1/3}}.
$$
Putting this together with \eqref{eq:term1} and \eqref{eq:term2}, from \eqref{eq:TermSplit} we obtain
$$
\left| \frac{1}{[K(\alpha):K]} \sum_{\sigma : K(\alpha)/K \hookrightarrow \C} \log |\sigma(\alpha)-\beta|_v - \log \max ( |\alpha|_v, |\beta|_v ) \right| \ll \frac{h(\beta)}{[K(\alpha):K]^{1/3}},
$$
and so taking $[K(\alpha):K]$ sufficiently large establishes \eqref{eq:squah}, completing the proof of Theorem~\ref{thm:main}.


\Address


\begin{thebibliography}{5}

\bibitem{Ba} A. Baker, \textit{Transcendental number theory}, Cambridge University Press, Cambridge, 1975.

\bibitem{BD} M. Baker and L. DeMarco, \textit{Preperiodic points and unlikely intersections}, Duke Math. J. \textbf{159} (2011), 1-29.

\bibitem{B} R. Benedetto, \textit{Dynamics in one non-archimedean variable}, American Mathematical Society, 2019.

\bibitem{BEG} A. B\'{e}rczes, J.-H. Evertse and K. Gy\"{o}ry, \textit{Effective results for hyper- and superelliptic equations over number fields}, Publ. Math. Bebrece \textbf{82} (2013), 727-756.

\bibitem{BG} E. Bombieri and W. Gubler, \textit{Heights in diophantine geometry}, Cambridge University Press, 2006.

\bibitem{BIR} M. Baker, S. Ih and R. Rumely, \textit{A finiteness property of torsion points}, Algebra Number Theory \textbf{2} (2008), no. 2, 217-248.


\bibitem{Co} J. Conway, \textit{Functions of one complex variable} (2nd edition), Springer-Verlag, New York, 1986.

\bibitem{FGS} A. Ferraguti, G. Micheli and R. Schnyder, \textit{On sets of irreducible polynomials closed by composition}, Arithmetic of Finite Fields, WAIFI 2016, {\it Lecture Notes in Computer Science}, {\bf 10064}, Springer, Cham (2017), 77--83.

\bibitem{FRL} C. Favre and J. Rivera-Letelier, \textit{Equidistribution quntitative des points de petite hauteur sur la droite projective}, Math. Ann. \textbf{335} (2006), no. 2, 311-361.

\bibitem{GMS}
D.  G\'omez-P\'erez,  L. M{\'e}rai and I.~E.~Shparlinski, 
\textit{On the complexity of exact counting of dynamically irreducible polynomials},
 J. Symb. Comp., \textbf{99} (2020), 231-241.

\bibitem{GN} D. Ghioca and K. Nguyen, \textit{The orbit intersection problem for linear
spaces and semiabelian varieties}, Math. Res. Lett., {\bf 24} (2017), 1263--1283.

\bibitem{GTZ} D. Ghioca, T. Tucker and M. Zieve, \textit{Linear relations between 
polynomial orbits}, Duke Math. J., {\bf 161} (2012), 1379--1410.

\bibitem{GTZ1}D. Ghioca, T. Tucker and M. Zieve, \textit{The Mordell-Lang question for endomorphisms of semiabelian varieties}, J. Theor. Nombres Bordeaux, {\bf  23} (2011), 645--666.

\bibitem{HBM} D. R. Heath-Brown and G. Micheli, \textit{Irreducible polynomials
over finite fields produced by composition of quadratics},
Revista Matem. Iber., \textbf{35} (2019), no. 3, 847-855.

\bibitem{IT} S. Ih and T. Tucker, \textit{A finiteness property for preperiodic points of Chebyshev polynomials}, Int. J. Number Theory, \textbf{6} (2010), no. 5, 1011-1025.

\bibitem{Kaw}  S. Kawaguchi,\textit{Canonical heights, invariant currents, and dynamical eigensystems of morphisms for line bundles}, J. Reine Angew. Math., {\bf 597} (2006), 135--173. 

\bibitem{K} S. Kawaguchi, \textit{Canonical heights for random iterations in certain varieties}, Int. Math. Res. Not. IMRN \textbf{7} (2007), Art. ID rnm 023, 33.

\bibitem{Mat} E. M. Matveev, \textit{An explicit lower bound for a homogeneous rational linear form in logarithms of algebraic numbers II}, Izv. Math., \textbf{64} (2000), 1217-1269.


\bibitem{OY} A. Ostafe and M. Young, \textit{On algebraic integers of bounded house and preperiodicity in polynomial semigroup dynamics}, Trans. Amer. Math. Soc. \textbf{373} (2020), 2191-2206

\bibitem{P} C. Petsche, \textit{S-integral preperiodic points for dynamical systems over number fields}, Bull. London Math. Soc., \textbf{40} (2008), no. 5, 749-758.


\bibitem{Sch4} A. Schinzel, \textit{Some unsolved problems on polynomials}, Matematicka Biblioteka, \textbf{25} (1963), 63-70.

\bibitem{Sch2} A. Schinzel, \textit{On the reducibility of polynomials and in particular of trinomials}, Acta Arith., \textbf{11} (1965), 1-34.

\bibitem{Sch} A. Schinzel, \textit{ Reducibility of lacunary polynomials III}, Acta Arith., \textbf{34} (1978), 227-266.

\bibitem{Sch3} A. Schinzel, \textit{Selected topics on polynomials}, University of Michigan Press, Ann Arbor, Michigan, 1982.

\bibitem{Si93} J. H. Silverman, {\it Integer points, Diophantine approximation, and iteration of
rational maps}, {Duke Math. J.}, {\bf 71} (1993), 793-829.

\bibitem{Yap} J. W. Yap, \textit{Uniform bounds on $S$-integral preperiodic points for power and Latt\`{e}s maps}, preprint, 2023.

\bibitem{Young} M. Young, \textit{Effective bounds on $S$-integral preperiodic points for polynomials}, preprint, 2023.

\bibitem{YThesis} M. Young, \textit{Effective integrality results in arithmetic dynamics}, PhD thesis, University of Cambridge, 2023, available at \url{https://www.repository.cam.ac.uk/items/0d5a5ef6-da9e-40e3-98ad-55ff21724924}.

\bibitem{Yu} K. R. Yu, \textit{$P$-adic logarithmic forms and group varieties III}, Forum Math. \textbf{19} (2007), 187-280.

\bibitem{YZ}  X. Yuan
and S. Zhang, \textit{The arithmetic Hodge index theorem for adelic line bundles}, Math. Ann., {\bf 367} (2017), 1123--1171.

\end{thebibliography}
\end{document}